\documentclass[10pt, reqno]{amsart}
\usepackage{amsmath, amsthm, amscd, amsfonts, amssymb, graphicx, color}
\usepackage[bookmarksnumbered, colorlinks, plainpages]{hyperref}
\hypersetup{colorlinks=true,linkcolor=red, anchorcolor=green, citecolor=cyan, urlcolor=red, filecolor=magenta, pdftoolbar=true}
\usepackage{mathrsfs}

\usepackage[OT2,T1]{fontenc}
\DeclareSymbolFont{cyrletters}{OT2}{wncyr}{m}{n}
\DeclareMathSymbol{\Sha}{\mathalpha}{cyrletters}{"58}

\newtheorem{theorem}{Theorem}[section]
\newtheorem{lemma}[theorem]{Lemma}

\newtheorem{corollary}[theorem]{Corollary}
\theoremstyle{definition}
\newtheorem{definition}[theorem]{Definition}
\newtheorem{example}[theorem]{Example}

\theoremstyle{remark}
\newtheorem{remark}[theorem]{Remark}
\numberwithin{equation}{section}

\begin{document}
\setcounter{page}{1}

\title[Arithmetic complexity]{Arithmetic complexity revisited}

\author[Nikolaev]
{Igor ~V. ~Nikolaev$^1$}

\address{$^{1}$ Department of Mathematics and Computer Science, St.~John's University, 8000 Utopia Parkway,  
New York,  NY 11439, United States.}
\email{\textcolor[rgb]{0.00,0.00,0.84}{igor.v.nikolaev@gmail.com}}

\dedicatory{In memory of Yuri ~I.  ~Manin}




\begin{abstract}
The arithmetic complexity counts  the number  of algebraically independent entries in the periodic  continued fraction 
$\theta=[b_1,\dots, b_N, \overline{a_1,\dots,a_k}]$.
If  $\mathscr{A}_{\theta}$ is a noncommutative torus corresponding to the rational elliptic curve $\mathscr{E}(K)$, 
then the rank  of $\mathscr{E}(K)$ is given by a simple formula $r(\mathscr{E}(K))= c(\mathscr{A}_{\theta})-1$, where $c(\mathscr{A}_{\theta})$
is the arithmetic complexity of $\theta$.
We prove that  $c(\mathscr{A}_{\theta})$ is equal to the dimension of the  Brock-Elkies-Jordan  
variety of $\theta$ introduced in \cite{BrElJo1}. 
 Following  Zagier and Lemmermeyer,  we evaluate the Shafarevich-Tate group  of $\mathscr{E}(K)$. 
\end{abstract}

\maketitle

\section{Introduction}
Noncommutative geometry studies  an interplay between  spatial forms
and algebras with non-commutative multiplication. To illustrate the idea, 
consider the noncommutative
torus  $\mathscr{A}_{\theta}$,  i.e. the   $C^*$-algebra  generated by the unitary operators $U$ and $V$ satisfying  the relation 
$VU=e^{2\pi i\theta}UV$,  where $\theta$ is a real constant [Rieffel  1990] \cite{Rie1}.
The  $\mathscr{A}_{\theta}$ is a non-commutative analog of  the coordinate ring of complex
elliptic curve $\mathscr{E}(\mathbf{C})$  [Manin 2004] \cite{Man1}.  Namely,   there exists a covariant functor
$F$ mapping  isomorphic elliptic curves $\mathscr{E}(\mathbf{C})$ to the Morita equivalent noncommutative tori   $\mathscr{A}_{\theta}$
\cite[Section 1.3]{N}.  A restriction of   $F$ to the  number fields  $K\subset\mathbf{C}$ corresponds to  
 the irrational roots  of  quadratic  polynomials  with  integer coefficients.  Such  roots  unfold  into 
  a  continued fraction of the period $(a_1,\dots, a_k)$ which we write in the form:
\begin{equation}\label{eq1.1}
\theta=[b_1,\dots, b_N, \overline{a_1,\dots,a_k}]. 
\end{equation}

\bigskip
Let $\mathscr{E}(K)\cong\mathbf{Z}^r \oplus \mathscr{E}_{tors}$ be the 
Mordell-Weil group  of the $K$-points of   $\mathscr{E}(\mathbf{C})$. 
To recast the rank $r\ge 0$ of $\mathscr{E}(K)$    in terms of  the continued fraction  (\ref{eq1.1}),   
an  arithmetic complexity $c(\mathscr{A}_{\theta})$  was used \cite[Section 6.2.1]{N}.  
 Roughly speaking, the  $c(\mathscr{A}_{\theta})$ counts 
the independent  entries $a_i$ and $b_j$ in (\ref{eq1.1}).  
Namely, the  $c(\mathscr{A}_{\theta})$ is equal to   the Krull dimension of 
 a projective variety $\mathscr{V}_E$  associated to the continued fraction (\ref{eq1.1}). 
 The  equation of $\mathscr{V}_E$  for $\theta=\{\sqrt{D} ~|~D>0 ~\hbox{square-free integer}\}$  was first studied 
 by   [Euler 1765] \cite{Eul1}, hence the notation.   
 The rank of elliptic curve $\mathscr{E}(K)$ and the arithmetic complexity of 
 $\mathscr{A}_{\theta}=F(\mathscr{E}(K))$  are linked  by a simple formula:
\begin{equation}\label{eq1.2}
r=c(F(\mathscr{E}(K)))-1. 
\end{equation}

\bigskip
An affine variety $V_{N,k}(\mathbf{C})$  associated to (\ref{eq1.1})  has been studied  by 
[Brock, Elkies \& Jordan 2019] \cite{BrElJo1}. 
In this remarkable paper, the authors prove that  the $V_{N,k}(\mathbf{C})$
fibers  over the Fermat-Pell conic $\mathscr{Q}$ [Brock, Elkies \& Jordan 2019] \cite[Theorem 3.5]{BrElJo1}.
The $\mathbf{Z}$-points  of   $V_{N,k}(\mathbf{C})$
correspond to  the  continued fractions (\ref{eq1.1}) and the reader is referred to Section 2.1  
for other  details. We call the $V_{N,k}(\mathbf{C})$  a {\it Brock-Elkies-Jordan variety}.  

\medskip
The aim of our note is to relate  geometry of the variety $V_{N,k}(\mathbf{C})$  with 
the rank $r$,   torsion $\mathscr{E}_{tors}$ and  the arithmetic complexity  $c(F(\mathscr{E}(K)))$.
 The maximal  connected component of  the Brock-Elkies-Jordan variety containing  a point 
 \linebreak
$(b_1,\dots,b_N; a_1,\dots, a_k)$ will be denoted by $V_{N,k}^0$.  (We refer the reader to Section 4.2
for quick examples  of the $V_{1,2}^0$.)  Let $\mathbb{A}^r$ be the $r$-dimensional affine space. 
 Our main results can be formulated as follows. 
\begin{theorem}\label{thm1.1}
$V_{N,k}^0$ is a fiber bundle over the  Fermat-Pell  conic  $\mathscr{Q}$
with the fibers $\mathbb{A}^r$ and the structure group  $\mathscr{E}_{tors}$. 
In particular,  the arithmetic complexity is given by the formula
\begin{equation}\label{eq1.3}
c(F(\mathscr{E}(K)))=\dim_{\mathbf{C}} V_{N,k}^0. 
\end{equation}
\end{theorem}
\begin{remark}
Theorem \ref{thm1.1} can be used to evaluate ranks of the elliptic curves
via formula (\ref{eq1.2}).  
As an example, we calculate such ranks in Section 4 using 
explicit formulas  for  the continued fractions (\ref{eq1.1}) with $N+k\le 3$
developed by [Brock, Elkies \& Jordan 2019] \cite[Sections 5-8] {BrElJo1}. 
 \end{remark}
Recall that the Shafarevich-Tate group $\Sha (\mathscr{E}(K))$ measures 
the failure of the Hasse principle for the elliptic curve $\mathscr{E}(K)$. 
Denote by $Cl ~(\Lambda)$ the (narrow)  class group of the order $\Lambda:=\mathbf{Z}+\mathbf{Z}\theta$
in the ring of integers of  the real quadratic  field $k=\mathbf{Q}(\theta)$. 
If $\Lambda$ is the maximal order, then  $Cl ~(\Lambda)$ is the class group 
of the field $k$.   If $\mathscr{Q}$ is a conic, then $\Sha (\mathscr{Q})\cong  Cl ~(\Lambda)\oplus Cl ~(\Lambda)$
 [Zagier 1991] \cite[Section 4.2]{Zag1} and 
[Lemmermeyer 2003] \cite[Theorem 11]{Lem1}.   
Applying this  formula to the fiber bundle $V_{N,k}^0$
over  the Fermat-Pell conic $\mathscr{Q}$, 
one gets the following  result. 
\begin{corollary}\label{cor1.3}
$\Sha (\mathscr{E}(K))\cong Cl ~(\Lambda)\oplus Cl ~(\Lambda)$. 
\end{corollary}
The article is organized as follows. Preliminary facts can be found 
in Section 2.  The proofs of Theorem \ref{thm1.1} and Corollary \ref{cor1.3} 
are given in Section 3.  Two examples  illustrating Theorem \ref{thm1.1}  are 
considered in Section 4.

\section{Preliminaries}
We briefly review the Brock-Elkies-Jordan varieties and arithmetic complexity of the noncommutative tori. 
For a detailed account we refer the reader to [Brock, Elkies \& Jordan 2019] \cite{BrElJo1}
and  \cite[Section 6.2.1]{N}, respectively.  

\subsection{Brock-Elkies-Jordan variety}
By an infinite continued fraction one understands an expression of the form:
\begin{equation}\label{eq2.1}
[c_1,c_2, c_3, \dots]:=c_1+\cfrac{1}{c_2+\cfrac{1}{c_3+\dots}} ~,
\end{equation}
where $c_1$ is an integer and $c_2,c_3,\dots$ are positive integers. 
The continued fraction (\ref{eq2.1}) converges to an irrational number and each irrational 
number has a unique representation by (\ref{eq2.1}).  The expression (\ref{eq2.1})
is called {\it $k$-periodic}, if $c_{i+k}=c_i$ for all $i\ge N$ and a minimal index $k\ge 1$.   
We shall denote the $k$-periodic continued fraction by 
\begin{equation}\label{eq2.2}
[b_1,\dots, b_N, \overline{a_1,\dots,a_k}],
\end{equation}
 where $(a_1,\dots, a_k)$ is the minimal period of (\ref{eq2.1}). 
The continued fraction (\ref{eq2.2}) converges to the irrational root of a quadratic 
polynomial 
\begin{equation}\label{eq2.3}
Ax^2+Bx+C\in \mathbf{Z}[x].
\end{equation}
Conversely,  the irrational root of any quadratic polynomial  (\ref{eq2.3}) has a
representation by  the continued fraction (\ref{eq2.2}). 
Notice that the following two continued fractions define the same irrational number: 
\begin{equation}\label{eq2.4}
[b_1,\dots, b_N, \overline{a_1,\dots,a_k}]= [b_1,\dots, b_N, a_1,\dots, a_k, \overline{a_1,\dots,a_k}]. 
\end{equation}
But it is well known,  that two infinite continued fraction with at most finite number of distinct 
entries must be related by the linear fractional transformation given  by a matrix  $\mathcal{E}\in GL_2(\mathbf{Z})$.
Therefore equation (\ref{eq2.4}) can be written in the form
\begin{equation}\label{eq2.5}
x={E_{11}x+E_{12}\over E_{21}x+E_{22}},
\end{equation}
where   $\mathcal{E}=(E_{ij})\in GL_2(\mathbf{Z})$ and $x=[b_1,\dots, b_N, \overline{a_1,\dots,a_k}]$. 
\begin{remark}\label{rmk2.1}
It is easy to see, that $x$ in  (\ref{eq2.5}) is  the root of  quadratic polynomial (\ref{eq2.3}) with $A=E_{21}, B=E_{22}-E_{11}$
and $C=-E_{12}$. 
\end{remark}
\begin{definition}\label{dfn2.2}
{\bf (\cite[Definition 3.1]{BrElJo1})}
The  Brock-Elkies-Jordan variety  $V_{N,k}(\mathbf{C})\subset \mathbb{A}^{N+k}$ is an affine variety over $\mathbf{Z}$  
defined by the three equations:
\begin{equation}\label{eq2.6*}
\left\{
\begin{array}{rl}
A[E_{22}-E_{11}](y_1,\dots,y_N,x_1,\dots,x_k) =& BE_{21}(y_1,\dots,y_N,x_1,\dots,x_k)\\
 -AE_{12}(y_1,\dots,y_N,x_1,\dots,x_k)  =&    CE_{21}(y_1,\dots,y_N,x_1,\dots,x_k)\\
-BE_{12}(y_1,\dots,y_N,x_1,\dots,x_k) =& C[E_{22}-E_{11}](y_1,\dots,y_N,x_1,\dots,x_k),\nonumber
\end{array}
\right.
\end{equation}
where $A,B,C$ are  constants and $E_{ij}\in\mathbf{Z}[y_1,\dots,y_N,x_1,\dots,x_k],$
see   Remark \ref{rmk2.1}. 
\end{definition}
\begin{remark}
Our notation  $V_{N,k}(\mathbf{C})$ corresponds to the variety  $V(\mathcal{B})_{N,k}$, 
where $\mathcal{B}$ is the multi-set of roots  [Brock-Elkies-Jordan 2021] \cite[Definition 3.1]{BrElJo1}.  
Such a change is justified,  since we focus on the continued fractions with the integer entries.  
\end{remark}

\bigskip
It is verified directly  from remark \ref{rmk2.1} and the equality $E_{11}E_{22}-E_{12}E_{21}=(-1)^k$,
that 
\begin{equation}\label{eq2.6}
CE_{21}^2-BE_{21}E_{22}+AE_{22}^2=(-1)^kA. 
\end{equation}
\begin{definition}
By the Fermat-Pell conic $\mathscr{Q}$ one understands the plane curve:
\begin{equation}\label{eq2.7}
C\mathbf{x}^2-B\mathbf{xy}+A\mathbf{y}^2=(-1)^kA. 
\end{equation}
\end{definition}

\begin{theorem}
{\bf (Brock, Elkies \& Jordan \cite{BrElJo1})}
The affine variety  $V_{N,k}(\mathbf{C})$ fibers over the Fermat-Pell conic $\mathscr{Q}$,
i.e. there exists a map  $\pi: V_{N,k}(\mathbf{C})\to \mathscr{Q}$,   such that 
\begin{equation}\label{eq2.8}
\pi(y_1,\dots,y_N,x_1,\dots,x_k)=(E_{21}, E_{22}). 
\end{equation}
\end{theorem}

\subsection{Arithmetic complexity}
The noncommutative torus $\mathscr{A}_{\theta}$ is said to have real multiplication,
if $\theta$ is the irrational root of a quadratic polynomial (\ref{eq2.3}). 
To distinguish this  case from the other values of $\theta$, we shall write
$\mathscr{A}_{RM}$. It is clear, that we can write $\theta$ as:
\begin{equation}\label{eq2.9}
\theta_d= {a+b\sqrt{d}\over c},
\end{equation}
where $a,b,c\in\mathbf{Z}$ and $d>0$ is a square-free integer. 
Consider a  family of the irrational numbers $\{\theta_x ~|~ x > 0\}$ 
of the form:
\begin{equation}\label{eq2.10}
\theta_x = {a+b\sqrt{x}\over c},   \quad a,b,c=Const \quad \hbox{and}   \quad x\in U_d,
\end{equation}
where $U_d$  is a set of the square-free integers containing $d$.
A system of the polynomial equations in the ring $\mathbf{Z}[y_1,\dots,y_N, x_1,\dots,x_k]$
is called Euler’s, if each $\{\theta_x ~|~ x\in U_d\}$ can be written as 
\begin{equation}\label{eq2.11}
\theta_x= [b_1(x),\dots, b_N(x),   \overline{a_1(x),\dots,a_k(x)}],
\end{equation}
where $a_i(x), b_j(x)\in \mathbf{Z}[x]$.
It is not hard to see, that (\ref{eq2.11}) is equivalent to the
equations $A=E_{21}, B=E_{22}-E_{11}, C=-E_{12}$
in the ring  $\mathbf{Z}[y_1,\dots,y_N, x_1,\dots,x_k]$,
since $\theta_x$ must satisfy condition (\ref{eq2.4}) and equations (\ref{eq2.5}),
see Remark \ref{rmk2.1}.   
Thus the Euler equations define an algebraic set
in the affine space $\mathbb{A}^{N+k}$. 
The Euler variety $\mathscr{V}_E$  is defined as  the projective closure of an irreducible affine variety containing 
the point $x = d$ of this set.
\begin{remark}
An immediate example of the Euler equations is given by formula (\ref{eq4.3}). 
Such equations are equivalent to the Brock-Elkies-Jordan equations in Definition \ref{dfn2.2},
see item (i) in the proof of Lemma \ref{lm3.1}.
 \end{remark}
\begin{remark}
The case $a=0$ and $c=1$ of (\ref{eq2.10}) was first studied by [Euler 1765] \cite{Eul1};  hence the name. 
\end{remark}
\begin{definition}
By  an arithmetic complexity $c(\mathscr{A}_{RM})$   of the noncommutative torus $\mathscr{A}_{RM}$
one understands the Krull dimension of the Euler variety   $\mathscr{V}_E$. 
\end{definition}
\begin{remark}\label{rmk2.6}
The $c(\mathscr{A}_{RM})$ counts the number of the algebraically independent entries in 
the continued fraction $\theta_d=[b_1,\dots, b_N,   \overline{a_1,\dots,a_k}]$.
\end{remark}
\begin{theorem}\label{thm2.7}
{\bf (\cite[Theorem 6.2.1] {N})}
The rank $r$ of the elliptic curve $\mathscr{E}(K)$ is given by the formula:
\begin{equation}\label{eq2.12}
r=c(\mathscr{A}_{RM})-1,
\end{equation}
 where $\mathscr{A}_{RM}=F(\mathscr{E}(K))$.
\end{theorem}

\section{Proofs}
\subsection{Proof of Theorem \ref{thm1.1}}
We shall split the proof in a series of lemmas.
\begin{lemma}\label{lm3.1}
The  connected component $V_{N,k}^0$ of the  Brock-Elkies-Jordan variety $V_{N,k}(\mathbf{C})$ is 
a fiber bundle over the Fermat-Pell conic $\mathscr{Q}$:
\begin{equation}\label {eq3.1}
\pi: V_{N,k}^0\to\mathscr{Q},
\end{equation}
such that  each fiber $\{\pi^{-1}(q) ~|~ q\in\mathscr{Q} \}$ 
is an $r$-dimensional affine space $\mathbb{A}^r$.
\end{lemma}
\begin{proof}
Roughly speaking, the idea is to take the projective closure of (\ref{eq3.1}) and 
apply \cite[Lemma 6.2.2]{N}.  
Such a lemma says the Euler variety $\mathscr{V}_E$ is a fiber bundle 
over the projective line $\mathbf{C}P^1$ with the fiber 
an abelian variety $A_E$. 

\medskip
(i) Let us show that
\begin{equation}\label {eq3.2}
Proj  ~V_{N,k}^0\cong \mathscr{V}_E,
\end{equation}
where $Proj$ is the projective closure of the affine variety $V_{N,k}^0$. 
Indeed, recall that the  $V_{N,k}^0$ is  defined by the coefficients 
$A, B$ and $C$ of the quadratic polynomial (\ref{eq2.3}). Using 
these coefficients, one can write (\ref{eq2.9}) in the form:
\begin{equation}\label {eq3.3}
\theta_d={-B+b\sqrt{d}\over 2A}, \quad\hbox{where}\quad  b^2d=B^2-4AC. 
\end{equation}
Thus the  $V_{N,k}^0$ defines  the Euler equations for the family $\theta_x$ and the 
Euler variety  $\mathscr{V}_E$, see Section 2.2.  Moreover, 
the equations in Definition  \ref{dfn2.2} must coincide with the Euler equations
for the family (\ref{eq2.11}).  We conclude from (\ref{eq3.3}) that   the $V_{N,k}^0$  is an irreducible 
affine variety containing the point $x=d$. Thus,  $Proj  ~V_{N,k}^0\cong \mathscr{V}_E$ by the definition 
of the Euler variety $\mathscr{V}_E$ in  Section 2.2.  Item (i) is proved. 

\bigskip
(ii)  To prove that $\{\pi^{-1}(q)\cong \mathbb{A}^r ~|~ q\in\mathscr{Q}\}$,
let $(\mathscr{V}_E, \mathbf{C}P^1,\pi', A_E)$ be the fiber bundle constructed 
in   \cite[Lemma 6.2.2]{N}.  Consider a morphism of the following fiber bundles:
\begin{equation}\label {eq3.4}
(V_{N,k}^0, \mathscr{Q}, \pi, X)\to (\mathscr{V}_E,  \mathbf{C}P^1,\pi', A_E),
\end{equation}
where $X$ is a fiber $\pi^{-1}(q)$ over $q\in\mathscr{Q}$. 
 Recall from item (i), that we have an isomorphism $Proj ~V_{N,k}^0\cong \mathscr{V}_E$. 
 Since  the genus of the Fermat-Pell conic  $\mathscr{Q}$ is equal to zero, 
 we conclude that $Proj~\mathscr{Q}\cong  \mathbf{C}P^1$. The fiber map 
 $\pi'$ is   the projective extension  of the map $\pi$.  

To calculate the fiber $X$ in (\ref{eq3.4}),  recall that the $A_E$ is an abelian 
variety of dimension $r$ \cite[Lemma 6.2.4]{N}. 
In particular, the $A_E$ is a compact algebraic group (group variety).  
Thus one gets a group morphism
\begin{equation}\label {eq3.5}
h:  X\to A_E,
\end{equation}
where $X$ is an affine group variety.
Since the fundamental group $\pi_1(X)$ is trivial, we conclude that $h$ is 
a covering map   with  $\ker (h)\cong \mathbf{Z}^r$, 
where $\mathbf{Z}^r$ is a discrete subgroup of $X$. Thus $X\cong \mathbb{A}^r$ is 
the $r$-dimensional affine space. Item (ii) is proved. 

\bigskip
To finish the proof of Lemma \ref{lm3.1}, we apply  
 \cite[Lemma 6.2.2]{N}.   A pullback of the 
 $(\mathscr{V}_E,  \mathbf{C}P^1,\pi', A_E)$ along the
 morphism (\ref{eq3.4}) 
 defines a fiber bundle $(V_{N,k}^0, \mathscr{Q}, \pi,\mathbb{A}^r)$. 
  Lemma \ref{lm3.1} is proved. 
\end{proof}

\medskip
\begin{lemma}\label{lm3.2}
The structure group of the fiber bundle  $(V_{N,k}^0, \mathscr{Q}, \pi, \mathbb{A}^r)$ is isomorphic to the abelian group 
$\mathscr{E}_{tors}$. 
\end{lemma}
\begin{proof}
Let $\mathscr{Q}$ be the Fermat-Pell conic defined by the equation (\ref{eq2.7}). 
Since $\theta_d$ in formula  (\ref{eq3.3}) is a real number,  the coefficients $A,B,C$ 
must satisfy the inequality $B^2-4AC>0$.  Thus equation   (\ref{eq2.7}) defines  a hyperbola. 
The completion of  $\mathscr{Q}$ by the point at infinity is homeomorphic 
to the unit circle $S^1=\mathscr{Q}\cup\{\infty\}$.  Denote by $(\tilde V_{N,k}^0, S^1, \tilde\pi, \mathbb{A}^r)$
 the corresponding one-point completion of $(V_{N,k}^0, \mathscr{Q}, \pi, \mathbb{A}^r)$  
 by a fiber $\mathbb{A}^r$ at the infinity. 

\smallskip
Consider the universal cover $\mathbf{R}$  of $S^1$ given by the formula
\begin{equation}\label{eq3.6}
t\mapsto e^{2\pi it}, \quad t\in\mathbf{R}.  
\end{equation}
We denote by $(\mathbb{A}^r\times \mathbf{R}, \mathbf{R}, \tilde\pi, \mathbb{A}^r)$ 
the pullback of the fiber bundle $(\tilde V_{N,k}^0, S^1, \tilde\pi, \mathbb{A}^r)$
along the map (\ref{eq3.6}).  Consider the integer points $\mathbf{Z}^r$ of the 
affine space $\mathbb{A}^r$.  It follows from \cite[Lemma 6.2.2]{N}, 
that $\mathscr{E}(K)$ is embedded into $\mathbb{A}^r\times \mathbf{R}$ 
according to the formula:
\begin{equation}\label{eq3.7}
\mathscr{E}(K)\cong\mathbf{Z}^r\oplus\mathscr{E}_{tors}\hookrightarrow
\mathbb{A}^r\times \mathbf{R}/m\mathbf{Z},
\end{equation}
where $m=|\mathscr{E}_{tors}|$.  Using formula (\ref{eq3.7}) and the map
\begin{equation}\label{eq3.8}
t\mapsto e^{2\pi imt}, \quad t\in\mathbf{R},   
\end{equation}
we conclude that the fiber bundle  $(\tilde V_{N,k}^0, S^1, \tilde\pi, \mathbb{A}^r)$
has the structure group 
\begin{equation}
\mathbf{Z}/m\mathbf{Z}\cong\mathscr{E}_{tors}.
\end{equation}
The same is true of  the fiber bundle $(V_{N,k}^0, \mathscr{Q}, \pi, \mathbb{A}^r)$, 
provided the point $\infty$ of conic $\mathscr{Q}$  is endowed with the fiber 
$\mathbb{A}^r$. Lemma \ref{lm3.2} is proved. 
\end{proof}

\begin{lemma}\label{lm3.3}
$c(\mathscr{A}_{\theta})=\dim_{\mathbf{C}} V_{N,k}^0$.
\end{lemma}
\begin{proof}
Consider the fiber bundle  $(V_{N,k}^0, \mathscr{Q}, \pi, \mathbb{A}^r)$. 
The dimension of a fiber bundle is the sum of dimensions of the base space and the 
fibers, i.e. 
\begin{equation}\label{eq3.10}
\dim_{\mathbf{C}} V_{N,k}^0=\dim_{\mathbf{C}} \mathbb{A}^r+\dim_{\mathbf{C}} \mathscr{Q}=
r+1. 
\end{equation}
Comparing  (\ref{eq3.10}) with (\ref{eq2.12}), we conclude that 
$c(\mathscr{A}_{\theta})=\dim_{\mathbf{C}} V_{N,k}^0$.
 Lemma \ref{lm3.3} is proved.
\end{proof}

\bigskip
Theorem \ref{thm1.1} follows from Lemmas \ref{lm3.1}-\ref{lm3.3}.

\subsection{Proof of corollary \ref{cor1.3}}
We shall split the proof in two lemmas.

\medskip
\begin{lemma}\label{lm3.4}
$\Sha (\mathscr{Q})\cong Cl ~(\Lambda)\oplus Cl ~(\Lambda)$. 
\end{lemma}
\begin{proof}
Consider the order $\Lambda=\mathbf{Z}+\mathbf{Z}\theta$ in the ring of
integers in the real quadratic field $k=\mathbf{Q}(\sqrt{D})$. 
Since $\theta$ is the root of quadratic equation (\ref{eq2.3}),  we conclude 
that $A=1, B=0$ and $C=-D$.  Using (\ref{eq2.7}) one can write the Fermat-Pell 
conic $\mathscr{Q}$  in the form: 
\begin{equation}\label{eq3.11}
\mathbf{y}^2 - D\mathbf{x}^2 =(-1)^k. 
\end{equation}

\medskip
By the analogy between the Birch-Swinnerton-Dyer conjecture for elliptic curves
and the Dirichlet class number formula observed in 
 [Zagier 1991]  \cite[beginning of Section 4.2]{Zag1} and  proved for conics in [Lemmermeyer, around 2003]  \cite[Theorem 11]{Lem1}, 
the Shafarevich-Tate
group $\Sha (\mathscr{Q})$ of the conic (\ref{eq2.7}) is given by the
formula: 
\begin{equation}\label{eq3.12}
\Sha (\mathscr{Q})\cong Cl ~(\mathscr{Q})\oplus Cl ~(\mathscr{Q}),
 \end{equation}
where $Cl ~(\mathscr{Q})$ is the class group of $\mathscr{Q}$. 
But $Cl ~(\mathscr{Q})$ is equal to the narrow class group of the order $\Lambda$. 
Thus $\Sha (\mathscr{Q})\cong Cl ~(\Lambda)\oplus Cl ~(\Lambda)$. 
Lemma \ref{lm3.4} is proved. 
\end{proof}

\begin{lemma}\label{lm3.5}
$\Sha (\mathscr{E}(K))\cong Cl ~(\Lambda)\oplus Cl ~(\Lambda)$. 
\end{lemma}
\begin{proof}
Using formula (\ref{eq3.7}),  one can identify 
the $K$-rational points of the elliptic curve $\mathscr{E}(K)\cong \mathbf{Z}^r\oplus\mathscr{E}_{tors}$ with the integer
points of  the fiber bundle  $(V_{N,k}^0, \mathscr{Q}, \pi, \mathbb{A}^r)$. 
  Since $\mathbb{A}^r$ is a rational variety,  the group $\Sha (\mathbb{A}^r)$
is trivial.  We conclude therefore, that the failure of 
the Hasse principle for $\mathscr{E}(K)$ occurs  only in the base space $\mathscr{Q}$. 
On the other hand, Lemma \ref{lm3.4} says that $\Sha (\mathscr{Q})\cong Cl ~(\Lambda)\oplus Cl ~(\Lambda)$.
Thus $\Sha (\mathscr{E}(K))\cong Cl ~(\Lambda)\oplus Cl ~(\Lambda)$. 
Lemma \ref{lm3.5} is proved. 
\end{proof}

\bigskip
Corollary \ref{cor1.3} follows from Lemma \ref{lm3.5}.

\bigskip
\begin{remark}
The formula $\Sha (\mathscr{E}(K))\cong Cl ~(\Lambda)\oplus Cl ~(\Lambda)$
can be proved in purely algebraic terms \cite[Corollary 1.3]{Nik1}.  
However, the approach based 
on the Lemmermeyer's Lemma \ref{lm3.4} and geometry of the fiber bundle 
 $(V_{N,k}^0, \mathscr{Q}, \pi, \mathbb{A}^r)$ seems to be more elegant. 
\end{remark}

\section{Examples}
We shall consider two examples illustrating Theorem \ref{thm1.1}. 
They correspond to the    variety $V_{1,2}(\mathbf{C})$ 
 [Brock, Elkies \& Jordan 2019] \cite[Section 8]{BrElJo1}.

\subsection{Brock-Elkies-Jordan  variety $V_{1,2}(\mathbf{C})$}
Consider the  variety $V_{N,k}(\mathbf{C})$
with  $N=1$ and $k=2$. 
According to  [Brock, Elkies \& Jordan 2019] \cite[Section 8]{BrElJo1}, in this case 
the variety $V_{1,2}(\mathbf{C})$ is defined by the system of equations: 
\begin{equation}\label{eq4.1}
\left\{
\begin{array}{rl}
A(x_1x_2-2y_1x_1) =& Bx_1\\
 A(y_1^2x_1-y_1x_1x_2-x_2) =&    Cx_1\\
B(y_1^2x_1-y_1x_1x_2-x_2) =& C(x_1x_2-2y_1x_1).
\end{array}
\right.
\end{equation}

\bigskip
\begin{theorem}\label{thm4.1}
{\bf ([Brock, Elkies \& Jordan 2019] \cite[Section 8]{BrElJo1})}
The Brock-Elkies-Jordan  variety $V_{1,2}(\mathbf{C})$ has:

\medskip
(i) one component $[y_1,\overline{0,0}]$,  if $A=0\ne B$  or $[y_1,\overline{0, x_2}]$, 
if $A=B=0$; 

\smallskip
(ii) two components  $[y_1,\overline{0,0}]$ and 
\begin{equation}\label{eq4.2}
\left[y_1, \overline{-{2Ay_1+B\over Ay_1^2+By_1+C}, ~{2Ay_1+B\over A} }\right],  
\quad\hbox{if} ~A\ne0; 
 \end{equation}

\smallskip
(iii) three components    $[y_1,\overline{0,0}]$, (\ref{eq4.2}) and 
$[-{B\over 2A},\overline{x_1,0}]$, if $B^2=4AC$ and $A\ne 0$. 
\end{theorem}

\subsection{Arithmetic complexity $c(F(\mathscr{E}(K)))$}
Consider the following two families of  elliptic curves $\mathscr{E}(K)$.
\begin{example}\label{ex4.2}
{\bf (\cite[Section 1.4]{N})}
Let $b\ge 1$ be an integer. Denote by $\mathscr{E}_{CM}$ an elliptic curve 
with complex multiplication by the  integers of  the imaginary quadratic field 
$\mathbf{Q}(i\sqrt{b^2+2})$.  Then $F(\mathscr{E}_{CM})=\mathscr{A}_{\sqrt{b^2+2}}$
\cite[Theorem 1.4.1]{N}.  It is easy to see, that 
\begin{equation}\label{eq4.3}
\sqrt{b^2+2}=[b,~\overline{b, ~2b}]. 
\end{equation}

\smallskip\noindent
Therefore  $N=1$ and $k=2$,  i.e. the  $V_{1,2}(\mathbf{C})$ is 
 the  Brock-Elkies-Jordan  variety  corresponding 
 to the continued fraction (\ref{eq4.3}). 
Recall that  a classification of the connected components of the 
variety $V_{1,2}(\mathbf{C})$ is provided  by Theorem \ref{thm4.1}. 
It is clear,  that the connected components $[y_1,\overline{0,0}]$,
$[y_1,\overline{0, x_2}]$ or $[-{B\over 2A},\overline{x_1,0}]$ 
in Theorem \ref{thm4.1} must be excluded. 
Thus  the component $V_{1,2}^0$ containing 
 continued fraction  (\ref{eq4.3}) is  given by the formulas  (\ref{eq4.2}). 
To calculate dimension of the  component $V_{1,2}^0$,  notice that  the 
substitution
\begin{equation}\label{eq4.4}
\left\{
\begin{array}{ccl}
y_1 &=& uv\\
 A  &=&    1\\
 B  &=& 0\\
C &=& -v(u^2v+1), 
\end{array}
\right.
\end{equation}
brings  component  (\ref{eq4.2}) to  the form:
\begin{equation}\label{eq4.5}
[uv,~\overline{2u, ~2uv}]=\sqrt{v(u^2v+1)}. 
\end{equation}

\smallskip\noindent
The restriction $2u=b, ~v=2$ shows  that our fraction (\ref{eq4.3}) is contained 
in the  component $V_{1,2}^0$ described by   (\ref{eq4.5}). 
Moreover, the  $V_{1,2}^0$ is the maximal component with such a property. 
Since the $V_{1,2}^0$ is parametrized by two complex variables $u$ and $v$, 
we conclude that 
\begin{equation}\label{eq4.6}
\dim_{\mathbf{C}} V_{1,2}^0=2=c(F(\mathscr{E}_{CM})). 
\end{equation}
\end{example}

\medskip
\begin{example}\label{ex4.3}
{\bf (\cite[Section 6.2.4.2]{N})}
Let $b\ge 3$ be an integer. Denote by $\mathscr{E}(\mathbf{Q})$ an elliptic 
curve defined by  the affine equation
\begin{equation}\label{eq4.7}
y^2=x(x-1)\left(x-{b-2\over b+2}\right). 
\end{equation}
It is known (\cite[Section 6.2.4.2]{N}) that 
$F(\mathscr{E}(\mathbf{Q}))=\mathscr{A}_{{1\over 2}(b+\sqrt{b^2-4})}$.
The reader can verify  that 
\begin{equation}\label{eq4.8}
{b+\sqrt{b^2-4}\over 2}=[b-1,~\overline{1, ~b-2}]. 
\end{equation}
Thus the  $V_{1,2}(\mathbf{C})$ is  the  Brock-Elkies-Jordan  variety corresponding to the continued fraction(\ref{eq4.8}).
 Clearly, the connected components $[y_1,\overline{0,0}]$,
$[y_1,\overline{0, x_2}]$ or $[-{B\over 2A},\overline{x_1,0}]$ in Theorem \ref{thm4.1} 
are excluded.  Therefore  the connected component $V_{1,2}^0$ 
 containing periodic continued fraction  (\ref{eq4.8})  is  given by formula (\ref{eq4.2}). 
Using the  substitution
\begin{equation}\label{eq4.9}
\left\{
\begin{array}{ccl}
y_1 &=& u-1\\
 A  &=&    1\\
 B  &=& -u\\
C &=& 1, 
\end{array}
\right.
\end{equation}
one can write (\ref{eq4.2}) in the form:
\begin{equation}\label{eq4.10}
[u-1,~\overline{1, ~u-2}]={u+\sqrt{u^2-4}\over 2}. 
\end{equation}
The continued fractions (\ref{eq4.8}) and  (\ref{eq4.10}) 
are identical after the substitution $u=b$. Thus   (\ref{eq4.8}) parametrizes the 
 component  $V_{1,2}^0$. 
Moreover, the  $V_{1,2}^0$ is the maximal component with such a property. 
Since the $V_{1,2}^0$ depends on  one complex variable $u$, 
we conclude that 
\begin{equation}\label{eq4.11}
\dim_{\mathbf{C}} V_{1,2}^0=1=c(F(\mathscr{E}(\mathbf{Q})). 
\end{equation}
\end{example}
\begin{remark}\label{rmk4.4}
Formula (\ref{eq4.11}) is valid only for the odd values of $b\ge 3$ in (\ref{eq4.8}).   
 Indeed, the substitution $b=2k$ brings the LHS of (\ref{eq4.8}) to the form $k+\sqrt{k^2-1}$. 
 The latter is a special case of the RHS of the family (\ref{eq4.5}) with $u=k$ and $v=-1$. 
 Thus the even values of $b\ge 3$ must be excluded. 
\end{remark}

\medskip
\subsection{Ranks of elliptic curves $\mathscr{E}(K)$}
The ranks of elliptic curves $\mathscr{E}_{CM}$ and $\mathscr{E}(\mathbf{Q})$ 
can be evaluated  using formulas (\ref{eq1.2}), (\ref{eq4.6}) and (\ref{eq4.11}). 
\begin{example}
Let the  $\mathscr{E}_{CM}$ be as in example \ref{ex4.2}. The formulas  (\ref{eq1.2}) and 
(\ref{eq4.6}) say that 
\begin{equation}\label{eq4.12}
r(\mathscr{E}_{CM})= c(F(\mathscr{E}_{CM}))-1=2-1=1. 
\end{equation}
\end{example}
\begin{remark}
The arithmetic of  elliptic curves $\mathscr{E}_{CM}$ has been studied 
by  [Gross 1980] \cite{G}.  In particular, it was proved that 
$r(\mathscr{E}_{CM})=1$, whenever $\mathscr{E}_{CM}$ (satisfying some additional assumptions \cite[Section 1.4.2]{N}) 
 is an elliptic curve with 
complex multiplication by the integers of the field $\mathbf{Q}(\sqrt{-p})$,
where $p\equiv 3 \mod 8$ is a prime number  [Gross 1980] \cite[Theorem 22.4.2]{G}.
The reader can verify, that letting  $b=2k+1$ in formula (\ref{eq4.3}) 
implies  $b\equiv 3 \mod 8$.   Therefore formula (\ref{eq4.12}) 
is a generalization of  the result of   [Gross 1980] \cite[Theorem 22.4.2]{G}.
We refer the reader to \cite[Table 1.1]{N} for more examples. 
\end{remark}
\begin{example}\label{ex4.7}
Let the  $\mathscr{E}(\mathbf{Q})$ be as in example \ref{ex4.3}. 
One gets using  (\ref{eq1.2}),  (\ref{eq4.11}) and  remark \ref{rmk4.4},   that 
\begin{equation}\label{eq4.13}
r(\mathscr{E}(\mathbf{Q}))= c(F(\mathscr{E}(\mathbf{Q})))-1=1-1=0. 
\end{equation}

\smallskip\noindent
Thus  the rank of of the generic fiber $\mathscr{E}(\mathbf{Q})$
of the family (\ref{eq4.7}) must be equal to zero, see also \cite{Nik2}. 
In other words, for the infinitely many odd integers $b\ge 3$ the rank 
of the elliptic curve  $\mathscr{E}(\mathbf{Q})$ is zero. 
\end{example}
\begin{remark}
The reader is advised against   testing  the validity of formula (\ref{eq4.13})
using a vast  database  of  the analytic values of $r(\mathscr{E}(\mathbf{Q}))$.
While such calculations give correct results for the generic fibers of (\ref{eq4.7}),  
they are misleading on the special fibers and  the values of $b$ lying outside an admissible 
set of values.   
 \end{remark}

\section*{Ethical statement}
The author refrain from misrepresenting research results which could damage the trust in the journal, 
the professionalism of scientific authorship, and ultimately the entire scientific endeavour.

\bibliographystyle{amsplain}


\end{document}